\documentclass[journal]{IEEEtran}
\usepackage[noadjust]{cite}
\usepackage{filecontents}
 \usepackage[flushleft]{threeparttable}
\usepackage{tikz}
 \usepackage{url}
\usepackage{cite}
\usepackage[utf8]{inputenc}
\usepackage{amssymb, amsmath, amsfonts}
\usepackage{ifthen}
\providecommand{\VersionLength}{long}
\newcommand{\ver}{\ifthenelse{\equal{\VersionLength}{long}}}
\newcommand{\nver}{\ifthenelse{\equal{\VersionLength}{short}}}
\renewcommand{\VersionLength}{long}

\usepackage{subfigure}
\usepackage{amsthm}
\usepackage{tikz}
\usepackage{url}
\usetikzlibrary{external,matrix,fit,external,arrows}
\usepackage{pgfplots}

	      \newif\ifcomment
	      \commenttrue
	      \ifcomment
	      \newcommand{\comments}[1]{\textcolor{blue}{#1}}
	      \else
	      \newcommand{\comments}[1]{}
	      \fi

\makeatletter
\newtheorem*{rep@theorem}{\rep@title}
\newcommand{\newreptheorem}[2]{%
\newenvironment{rep#1}[1]{%
 \def\rep@title{#2 \ref{##1}}%
 \begin{rep@theorem}}%
 {\end{rep@theorem}}}
\makeatother

\newtheorem{proposition}{Proposition}
\newtheorem{theorem}{Theorem}
\newtheorem{definition}{Definition}
\newtheorem{lemma}{Lemma}

\newtheorem{corollary}{Corollary}

\newreptheorem{theorem}{Theorem}
\newreptheorem{lemma}{Lemma}
\newreptheorem{proposition}{Proposition}

\newlength\figureheight
\newlength\figurewidth
\ifnum\pdfshellescape=1
\fi

\DeclareMathOperator*{\diag}{diag}     % argmin
\DeclareMathOperator*{\tr}{tr}     % trace

\newcommand{\I}{{\mathcal I}}
\newcommand{\mi}{{\mathcal J_1}}
\newcommand{\mj}{{\mathcal J_2}}
\newcommand{\mall}{{\mathcal J_i}}
\newcommand{\J}{{\mathcal J}}
\newcommand{\K}{{\mathcal K}}

\newcommand{\plant}{G}
\newcommand{\pwe}{F}

\newcommand{\gam}{\rho}

\newcommand{\part}{\mathcal D}
\newcommand{\s}{z}
\newcommand{\1}{\delta}
\newcommand{\act}{\max_{\mi, \mj \in \mathcal V } \left( \part^{\mi,\mj}_{2,2} \right)^2}
\newcommand{\op}{\mathcal P_\1}
\newcommand{\opd}{\mathcal D_\1}

\newcommand{\ie}{\textit{i.e.}}
\newcommand{\G}{\mathcal B}
\newcommand{\BEQ}{\begin{eqnarray*}}
  \newcommand{\EEQ}{\end{eqnarray*}}
  \newcommand{\BEQL}{\begin{eqnarray}}
    \newcommand{\EEQL}{\end{eqnarray}}
    \newcommand{\BDEF}{\begin{definition}}
      \newcommand{\EDEF}{\end{definition}}
      \newcommand{\BTHM}{\begin{theorem}}
	\newcommand{\ETHM}{\end{theorem}}
	\newcommand{\BLEM}{\begin{lemma}}
	  \newcommand{\ELEM}{\end{lemma}}
	  \newcommand{\BPF}{\begin{proof}}
	    \newcommand{\EPF}{\end{proof}}
	    \newcommand{\BCO}{\begin{corollary}}
	      \newcommand{\ECO}{\end{corollary}}
	      
	      \newcommand{\inr}{\mathbb{R}}

	      \newcommand{\comset}{\mathcal {V}}

\ifCLASSINFOpdf
\else
\fi
\usepackage{mathtools}
\usepackage{algpseudocode}
\usepackage{algorithm}
\usepackage{comment}

\usepackage{color}
\usepackage{ifthen}

\begin{document}

\title{Attack-Resilient $\mathcal H_2$, $\mathcal H_\infty$, and $\ell_1$ State Estimator}

\author{Yorie~Nakahira,~\IEEEmembership{Student Member,~IEEE,}
       and~Yilin~Mo,~\IEEEmembership{Member,~IEEE,}
        % <-this % stops a space
\thanks{Y. Nakahira was with the Department
of Computing and Mathematical Sciences, California Institute of Technology, Pasadena,
CA, 91125 USA, e-mail: ynakahir@caltech.edu, website: http://users.cms.caltech.edu/\textasciitilde ynakahir/.}% <-this % stops a space
\thanks{Y. Mo are with the Department
of Electrical and Electronic Engineering, Nanyang Technological University, 639798 Singapore, e-mail: ylmo@ntu.edu.sg, website: http://yilinmo.github.io/}% <-this % stops a space
\thanks{Manuscript received in June 2017.}}


\maketitle

\begin{abstract}

This paper considers the secure state estimation problem for noisy systems in the presence of sparse sensor integrity attacks. We show a fundamental limitation: that is, $2\rho$-detectability is necessary for achieving bounded estimation errors, where $\rho$ is the number of attacks. This condition is weaker than the $2\rho$-observability condition typically assumed in the literature. Conversely, we propose a real-time state estimator that achieves the fundamental limitation. The proposed state estimator is inspired by robust control and FDI: that is, it consists of local Luenberger estimators, local residual detectors, and a global fusion process. We show its performance guarantees for $\mathcal H_2$, $\mathcal H_\infty$, and $\ell_1$ systems. Finally, numerical examples show that it has relatively low estimation errors among existing algorithms and average computation time for systems with a sufficiently small number of compromised sensors.

\end{abstract}

\begin{IEEEkeywords}
Secure state estimation, fault tolerance, sparse sensor integrity attacks, Luenberger observer, robust control
\end{IEEEkeywords}

% \fi
\IEEEpeerreviewmaketitle

%%%%%%%%%%%%%%%%%%%%%%%%%%%%%%%%%%%%%%%%%%%%%%%%%%
%\vspace{-5mm}
\section{Introduction}

Fault tolerance in Cyber-physical Systems (CPSs) is of great importance~\cite{securecontrol2008,cardenas2009challenges,rajkumar2010cyber,sridhar2012cyber,datta2016accountability}. For example in the power system, false data injection can introduce errors in state estimation and provide financial gains for attackers~\cite{henrik2010, Xie2011,liu2009}. In flights, autonomous vehicles, and the Internet of Things, manipulations in software and sensing can cause human injury and economic damage~\cite{farwell2011stuxnet,maurer2016autonomous,nobles2016cyber}. Motivated by these security issues, this paper studies the secure estimation problem for noisy systems in the presence of sensor integrity attacks. 

%For the secure estimation problem in static systems, robust estimators are extensively studied in the literature. Common robust estimators include the M-estimator, L-estimator, and R-estimator~\cite{Kassam1985,robust2006,robust2009}, and they are used to account for sensor integrity attacks in \cite{mo2015secure}. For the secure estimation problem in dynamical systems, robust control provides tools to deal with noise in estimation and control~\cite{dahleh1994control,zhou1996robust}. Although robust control typically assumes that system disturbances are bounded or follow well-defined distributions, such assumptions may not be valid for sensor faults caused by intelligent attackers~\cite{mo2015secure,datta2016accountability}. Another area of research that addresses the secure estimation problem is fault detection and isolation (FDI)~\cite{gertler1998fault,venkatasubramanian2003review,isermann2006fault,chen2012robust}. One common approach of FDI for a linear system under sensor integrity attacks is to construct \textit{residuals} that take non-zero values only in the presence of faults (see~\cite{patton2013issues} and references therein). The generation of such residuals is possible only when a fault is separable from other disturbances and modeling uncertainties, which requires certain kinds of system observability. 

For the secure estimation problem in static systems, robust estimators are extensively studied in the literature. Common robust estimators include the M-estimator, L-estimator, and R-estimator~\cite{Kassam1985,robust2006,robust2009}, and they are used to account for sensor integrity attacks in \cite{mo2015secure}. For the secure estimation problem in dynamical systems, robust control provides tools to deal with noise in estimation and control~\cite{dahleh1994control,zhou1996robust}. Although robust control typically assumes that system disturbances are bounded or follow well-defined distributions, such assumptions may not be valid for sensor faults caused by intelligent attackers~\cite{mo2015secure,datta2016accountability}. Fault detection and isolation (FDI) also provides methods for identifying and pinpointing faults in sensors~\cite{gertler1998fault,venkatasubramanian2003review,isermann2006fault,chen2012robust}. One common approach of FDI for linear dynamical systems under sensor integrity attacks is to construct \textit{residuals} that take non-zero values only in the presence of faults (see~\cite{patton2013issues} and references therein). The generation of such residuals is possible only when a fault is separable from normal disturbances and modeling uncertainties, which requires certain kinds of system observability.

When attackers can change the measurements of a limited number of sensors in large-scale systems, sensor attacks can be modeled as \emph{sparse} but unbounded disturbances. For sparse sensor integrity attacks, recent literature has studied the fundamental limitation and achievable performance to identify the attacks and estimate the system states. Fawzi et al. show that if $\rho$ sensors are compromised, then $2 \rho$-observability (\ie, the system remains observable after removing any set of $2 \rho$ sensors) is necessary to guarantee perfect attack identification and accurate state estimation for noiseless systems~\cite{Fawzi2014}. The authors further propose to solve a $\ell_0$ problem to achieve accurate state estimation under the assumption of $2 \rho$-observability. This work is generalized to noisy systems by Pajic et al.~\cite{pajic2017design,pajic2017attack}. Shoukry et al. propose to use the Satisfiability Modulo Theory (SMT) solver to harness the complexity in secure estimation~\cite{shoukry2017secure}. However, the worst case complexity for the $\ell_0$ optimization and that of the SMT solver are combinatorial. Moreover, these estimators also have delays, which may cause performance degradation when used for real-time control. To transform the problem into a convex program, Fawzi et al. and Mo et al. propose to use optimization based methods~\cite{Fawzi2014,mo2016secure}. To address the estimation delays, various Luenberger-like observers are proposed~\cite{nakahira2015dynamic,shoukry2016event,mo2016secure,lu2017secure,chong2015observability,shoukry2017secure}. It is worth noticing that the estimators proposed in \cite{pajic2017design,pajic2017attack,shoukry2016event,mo2016secure,lu2017secure,chong2015observability,shoukry2017secure} require the assumption of $2 \rho$-observability or stronger to guarantee accurate attack identification and secure state estimation. 

In this paper, we consider the fundamental limitation and achievable performance to achieve fault-tolerant estimation. By fault-tolerant estimation, we refer to achieving bounded estimation errors. Compared with fault identification, fault-tolerant estimation requires relaxed assumptions and accounts for potentially non-detectable and non-identifiable attacks in noisy systems. We prove that a necessary condition to achieve fault-tolerant estimation under $\rho$ compromised sensors is that the system is $2 \rho$-detectable (the system needs to remain detectable after removing any set of $2 \rho$ sensors). This necessary condition suggests that, if a system has many stable modes, then the number of sensors required to achieve fault-tolerant estimation is much smaller than that to achieve fault identification. Conversely, we propose a secure state estimator that guarantees bounded estimation error under the assumption of $2\rho$-detectability. The proposed state estimator is inspired by robust control and FDI: that is, it consists of the local Luenberger estimators, the local residual detectors, and a global fusion process. A preliminary version of this paper was presented at the 2015 IEEE Conference on Decision and Control, deriving the worst-case estimation errors in the $\ell_1$ system~\cite{nakahira2015dynamic}. This paper extends the result of~\cite{nakahira2015dynamic} to the $\mathcal H_2$ system and the $\mathcal H_\infty$ system. To the best of our knowledge, our paper is the first to show that a mixture of two-norm bounded and sparse-unbounded input can produce \textit{two-norm} bounded output. Finally, numerical examples show that the proposed state estimator has relatively low estimation errors among existing algorithms and average computation time for systems with a sufficiently small number of compromised sensors.

\section{Preliminary}

\subsection{Notations} The set of natural numbers is denoted $\mathbb N$, the set of non-negative integers is denoted $\mathbb Z_+$, the set of real numbers is denoted $\mathbb R$, the set of non-negative real numbers is denoted $\mathbb R_+$, and the set of complex numbers is denoted $\mathbb C$. 
The cardinality of a set $S$ is denoted $| S |$. 
%For a complex number $z \in \mathbb C$, its real part is denoted $\text{re}(z)$, its imaginary part is denoted $\text{im}(z)$, its absolute value is denoted $\text{abs}(z)$, and its complex conjugate is denoted $\bar z$. 
A sequence $\{x(t)\}_{t \in Z_+}$ is abbreviated by the lower case letter $x$, and the truncated sequence from $t_1$ to $t_2$ is denoted $x(t_1:t_2)$. Let $\|x\|_0 = | \{ i: \exists t \text{ s.t. } x_i(t) \neq 0 \}|$ denote the number of entries in $x$ that take non-zero values for some time. 
%A $n \times n$ identify matrix is denoted $I_n$ or abbreviated to $I$. The cardinality of a discrete set $\mathcal I$ is denoted $\text{card}(\mathcal I)$.
%\subsection{Signal Norms} 
The infinity-norm of a sequence $x \in \mathbb R^n$ is defined as $\|x\|_\infty \triangleq \sup_{t} \max_{i} |x_i(t)|,$ and the two-norm of a sequence $x$ is defined as $\|x\|_2 \triangleq (\sum_{t=0}^\infty \sum_{i = 1}^n |x_i(t)|^2 )^{1/2}.$ 
Similarly, the norms of a truncated sequence $x(0:T)$ are defined as 
$
\|x(t_1:t_2)\|_\infty \triangleq \max_{t_1 \leq t \leq t_2}\max_i |x_i(t)| $ and $
\|x(t_1:t_2)\|_2 \triangleq (\sum_{t=t_1}^{t_2}\sum_{i = 1}^n |x_i(t)|^2 )^{1/2}.
$
Let $\ell_\infty$ be the space of sequences with bounded infinity-norm, and $\ell_2$ be the space of sequences with bounded two-norm.

\subsection{LTI Systems and System Norms} Let $\plant$ be the following discrete-time linear time-invariant (LTI) system:
\begin{align}
\label{eq:LTI}
&x(t+1) = A x(t) + B w(t) , & y(t) = C x(t) + D w(t),
\end{align}
 with the initial condition $x(0) = 0$, system state $x(t) \in \mathbb R^n$, system input $w(t) \in \mathbb R^l$, and system output $y(t) \in \mathbb R^m$. The transfer matrix of the system is
\BEQ
\plant = \left[
  \begin{array}{c|c}
    A & B \\
    \hline
    C &  D
  \end{array}
  \right] . 
\EEQ
The transfer function of the system is
$\hat \plant(z) = B ( z I - A )^{-1} C + D$. For $\ell_p$ system input and $\ell_q$ system output, the system norm (namely, \textit{induced norm}) is given by
    \begin{align}
    \label{eq:induced_norm}
      \|\plant\|_{p \rightarrow q} \triangleq \sup_{\| w\|_p \neq 0} \frac{ \|y \|_q}{\| w\|_p}.
    \end{align}
In particular, the induced-norms for $(p,q) = (2, 2), (2, \infty), (\infty, \infty)$ are given by $\mathcal H_\infty$, $\mathcal H_2$, and $\mathcal L_1$ norms, respectively. These induced-norms are bounded when $A$ is strictly stable (\ie, all the eigenvalues of $A$ is in the open unit circle). See \cite{dahleh1994control,zhou1996robust} for further details.

\ver{
In particular, the induced-norms for $(p,q) = (2, 2), (2, \infty), (\infty, \infty)$ are respectively $\mathcal H_\infty$, $\mathcal H_2$, and $\mathcal L_1$ norms, which are defined as
\begin{align}
&\|\plant\|_{2 \rightarrow \infty}  = \|\plant\|_{2} = \int^\pi_{-\pi} \tr ( \hat \plant( e^{i \theta}  ) \hat \plant^T ( e^{-i \theta} ) ) d \theta / 2 \pi\\
&\|\plant\|_{2 \rightarrow 2}  = \| \plant\|_{\infty} = \text{ess} \sup_{e^{i \theta}} \sigma_{\max} \hat \plant(e^{i \theta})\\
&\|\plant\|_{\infty \rightarrow \infty}  = \| \plant\|_{1} = \underset{1\leq i \leq n}{\text{max}} \sum_{j = 1}^{l} \sum^{\infty}_{t = 0} |h_{ij} (t)|
\end{align}
where $h_{ij}$ is the impulse response from $w_j(t)$ to $y_i(t)$. These induced-norms are bounded when $A$ is strictly stable (\ie, all the eigenvalues of $A$ is in the open unit circle). See \cite{dahleh1994control} for further details. The induced norms allow us to study  
}

%A system is detectable if all its unstable modes can be estimated using the system output. %\begin{definition}
%The pair $(A,C)$ is detectable if $A+KC$ is stable for some matrix $K$. 
%\end{definition}

%Detectability implies that, for any eigenvector $v \in \mathbb C^n$ and corresponding eigenvector $z \in \mathbb C$ of $A$ (\ie, $A v = z v$) such that $|z|\geq 1$, condition $Cv \neq 0$ holds \cite{zhou1996robust}. 
%The pair $(C,A)$ is detectable if and only if there exists some $z \in \mathbb C$, $v \in \mathbb C^n$ that satisfies $A v = z v$ and $C v = 0$.  

%In this section, we define the detectability of an LTI system and explain its consequences.  
To construct a linear state estimator with bounded estimation errors, the LTI system \eqref{eq:LTI} is required to be detectable, \ie, there exist some matrix $K$ such that $A+KC$ is stable. Given the matrix $K$ such that $A+KC$ is strictly stable, we can construct a linear estimator
\BEQL
\label{eq:LTIest}
\hat x(t+1) = A \hat x(t) - K ( y(t) - C\hat x(t)),\,\hat x(0)=0.
\EEQL
We define its estimation error $e$ and residual vector $r$ as 
\begin{align*}
  &e(t) \triangleq x(t) - \hat x(t), &r(t) \triangleq y(t) - C\hat x(t),
  \label{eq:deferrorandresidue}
\end{align*}
respectively. The signals $e$ and $r$ satisfy the following dynamics 
\begin{align*}
&  e(t+1) = (A + K C)  e(t) + (B  + KD) w(t), & e(0) = 0\\
&  r(t) = C e(t) + D w(t),
\end{align*}
Thus, the LTI system from $w$ to $e$, $E(K)$, and the LTI system from $w$ to $r$, $\pwe(K)$, are respectively given by
\BEQL
\label{eq:def_Ek}
&E(K) =
\left[
\begin{array}{c|c}
  A + K C & B+ K D\\
  \hline
  I &  0
\end{array}
\right] \\
\label{eq:def_Gk}
&\pwe(K) =
\left[
\begin{array}{c|c}
  A + K C & B+ K D \\
  \hline
  C &  D
\end{array} \right].
\EEQL
Because $A + K C$ is strictly stable, both $E(K) $ and $\pwe(K)$ have bounded induced-norms. The induced-norms upper-bound the values of $\| e \|_q$ and $\| r \|_q$ as follows.
\BLEM
\label{lemma:nec_original}
If $\| w \|_p \leq 1$, then the estimation error $e$ and residual vector $r$ satisfy 
\begin{align}
%\label{eq:errorbound}
&&\| e \|_q \leq \| E(K) \|_{p \rightarrow q},
&&\| r \|_q \leq \| \pwe(K)\|_{p \rightarrow q}.
%\label{eq:residuebound}
\end{align}
\ELEM

\section{Problem Formulation}
\label{sec:problem}

We study the secure state estimation problem in the presence of sensor attacks. Consider the discrete-time LTI system:
\begin{equation}
\label{eq:system}
\begin{aligned}
&x(t+1) = A x(t) + B w(t) , & x(0)=0\\
&y(t) = C x(t)  + D w(t) +  a(t),
\end{aligned}
\end{equation}
where $x(t) \in \inr^{n}$ is the system state, $w(t) \in \inr^l$ is the input disturbance, $y(t) \in \inr^{m}$ is the output measurement, and $a(t) \in \inr^{m}$ is the bias injected by the adversary (we call $a(t)$ the \textit{attack}). This system is illustrated in Fig. \ref{fig:robustestimationwithattack}. The time indices $t \in \mathbb Z_+$ are non-negative integers and start from zero. Without loss of generality, we assume that the disturbance matrix $B$ has full row rank (otherwise we can perform the Kalman decomposition and work on the controllable space of $(A,B)$). Each sensor is indexed by $i \in \{ 1, \cdots, m\}$ and produces measurement $y_i(t)$, which jointly comprises the measurement vector $y(t) = [y_1(t),\dots,y_m(t)]^T$. Sensor $i$ is said to be \textit{compromised} if $a_i(t) \neq 0$ at some time $t \in \mathbb Z_+$ and is said to be \textit{benign} otherwise. The maximum number of sensors that the attacker can compromise is $\gam$, \ie,  
\begin{align}
\label{eq:attack}
\|a\|_0 \leq \gam .
\end{align}
If $a(t)$ satisfies \eqref{eq:attack}, then we say that it is $\gam$\textit{-sparse}. 
Let $\mathcal S \triangleq \{1,\dots,m\}$ denote the set of all sensors, $\mathcal C\subset \mathcal S$ denote the set of compromised sensors, and $\G \triangleq \mathcal S\backslash \mathcal C$ denote the set of benign sensors. The set $\mathcal C$ is assumed to be unknown.\footnote{Take the setting of \cite{pajic2014robustness} for example. When the system is noisy, the optimization problem $\min_{x_t \in \inr^n } \| [y(t)^T, \cdots , y( t+n-1)^T ]^T - \mathcal O x_t \|$ ($\mathcal O$ is the observability matrix) may not give correct set of compromised sensors $\{ i : \exists t, a_i (t) \neq 0\}$.} A causal state estimator is an infinite sequence of functions $\{f_t\}$ where $f_t$ is a mapping from all output measurements to a state estimate: 
\begin{equation}
\label{def:state_estimator}
\hat x(t) = f_t( y(0:t-1) ).
\end{equation}
The estimation error of \eqref{def:state_estimator} is defined as the difference between the system state and the state estimate:\footnote{Although abbreviate it as $e(t)$, the estimation error is also a function of disturbance $w$, attack $a$, and the estimator $\{f_t\}$.  }
\begin{equation}
e(t) \triangleq x(t) - \hat x(t) . 
\end{equation}

   \begin{figure}[ht]
    \begin{center}
      \begin{tikzpicture}[>=stealth',
	  box/.style={rectangle, draw=blue!50,fill=blue!20,rounded corners, semithick,minimum size=7mm},
	  point/.style={coordinate},
	  every node/.append style={font=\small}
	]
	\matrix[row sep = 5mm, column sep = 4mm]{
	  %first row
	  \node (B) [box] {$B$};
	  &\node (noise) {$w$};
	  &;
	  &\node (D) [box] {$D$};
	  &;
	  &;
	  &\\
	  %second row
	  \node(plant) [box] {$(zI-A)^{-1}$};
	  & \node (p1) [point] {};
	  &\node (C) [box] {$C$};
	  & \node (plus1) [circle,draw,inner sep=2pt] {};
	  &\node (estimator) [box] {Estimator};
	  & \node (plus2) [circle,draw,inner sep=2pt] {};
	  &\node (error) {$e$};\\
	  %third row
	  ;
	  &;
	  &;
	  &\node (attack) {a};
	  &;
	  &;
	  &\\
	  %fourth row
	  ;
	  &\node (p2) [point] {};
	  &;
	  &;
	  &;
	  &\node (p3) [point] {};
	  &\\
	};
	\draw [semithick] (plus1.north)--(plus1.south);
	\draw [semithick] (plus1.east)--(plus1.west);
	\draw [semithick] (plus2.north)--(plus2.south);
	\draw [semithick] (plus2.east)--(plus2.west);
	\draw [semithick,->] (noise)--(B);
	\draw [semithick,->] (B)--(plant);
	\draw [semithick,->] (plant)-- node[midway,above]{$x$} (C);
	\draw [semithick,->] (p1)--(p2)--(p3)--node[very near end,right]{$+$} (plus2);
	\draw [semithick,->] (C)--(plus1);
	\draw [semithick,->] (noise)--(D);
	\draw [semithick,->] (D)--(plus1);
	\draw [semithick,->] (attack)--(plus1);
	\draw [semithick,->] (plus1)--node[midway,above]{$y$} (estimator);
	\draw [semithick,->] (estimator)--node[midway,below]{$\hat x$}
	node[very near end,above] {$-$} 
	(plus2);
	\draw [semithick,->] (plus2)--(error);
      \end{tikzpicture}
    \end{center}
    \caption{Diagram of the Estimation Problem in Adversarial Environment. }
    \label{fig:robustestimationwithattack} 
  \end{figure}
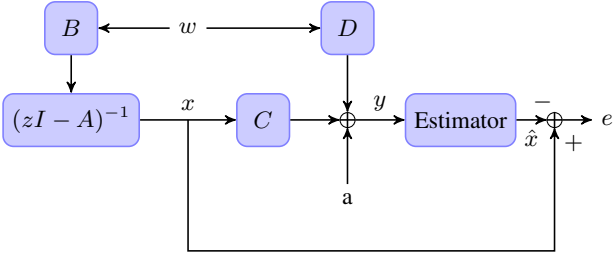

We consider the input containing a mixture of a $p$-norm bounded disturbance and a $\gam$-sparse attack and study the following worst-case estimation error in $q$-norm:
\begin{equation}  \label{eq:defworstcaseerror}
 \sup_{\|w\|_p\leq 1,\,\|a\|_0\leq \gam} \|e\|_q,
\end{equation}
where
 \begin{align}
(p, q) = (2,2),(2,\infty), (\infty,\infty). 
\end{align}
We consider \eqref{eq:defworstcaseerror} instead of attack isolation because the attack on a noisy system may not be correctable in the sense defined in \cite{Fawzi2014,pajic2014robustness}.

\begin{definition}
An causal state estimator $\{f_t\}$ is said to be $\epsilon$-resilient to attack if its worst-case estimation error satisfies $\sup_{\|w\|_p\leq 1,\,\|a\|_0\leq \gam} \|e\|_q < \epsilon$, where $\epsilon$ is a positive and finite scalar. 
\end{definition}
When the estimator is $\epsilon$-resilient for some finite $\epsilon >0$, then we say the state estimator is resilient to attack. The goal of this paper is to study the design problem of a resilient state estimator $\{ f_t\}$. Towards that end, we first show a fundamental limitation for the existence of a resilient estimator (Section \ref{sec:limitations}), and we then propose a resilient estimator (Section \ref{sec:design}) and analyze the estimation errors (Section \ref{sec:est_analysis}).

\section{Necessary and Sufficient Conditions for Resilience to Attack}
\label{sec:main}
In this section, we first provide a necessary condition for the existence of a resilient estimator and then, assuming the necessary condition, propose a resilient estimator. We first define some notation that will be used later. 

\begin{definition}[Projection map]
Let $e_i$ be the $i_{th}$ canonical basis vector of the space $\mathbb R^m$ and $\mathcal I = \{i_1,\dots,i_{m'}\}\subseteq \mathcal S$ be an index set with carnality $m' (\leq m)$. We define the projection map $P_{\I}: \mathbb R^m \rightarrow \mathbb R^{m'}$ as 
  \begin{align}
  \label{eq:proj_matrix}
    P_\I = \begin{bmatrix}
      e_{i_1}&\dots&e_{i_{m'}}
    \end{bmatrix}^T\in \mathbb R^{m' \times m}. 
  \end{align}
  \end{definition}
  \noindent Using $P_{\mathcal I}$ in \eqref{eq:proj_matrix}, the measurements of the set of sensors $\mathcal I \subset \mathcal S$ can be written as
  \begin{align*}
   y_{\I}(t) \triangleq P_{\mathcal I} y(t) \in \mathbb R^{m'}.
  \end{align*}
Similarly, the measurement matrix and the sensor noise matrix corresponding to the set of sensors $\mathcal I $ can be respectively written as 
\begin{align*}
&C_{\mathcal I} \triangleq P_\I C, &D_{\mathcal I} \triangleq P_\I D.
\end{align*}

\subsection{Necessary Condition for Resilience to Attack}
\label{sec:limitations}

In this section, we give a fundamental limitation for achieving bounded worst-case estimation errors. 

\begin{definition}
The system \eqref{eq:system} is said to be $\chi$-detectable if $(A, C_\K)$ is detectable for any set of sensors $\K \subset \mathcal S $ with cardinality $|\K| = m - \chi$.
\end{definition}

\begin{theorem}\label{thm:nonexistence}
If system \eqref{eq:system} is not $2 \rho$-detectable, then there is no state estimator $\{ f_t\}$ that is $\epsilon$-resilient to attack for any finite $\epsilon > 0$. 
\end{theorem}

\nver{\begin{proof}
See the extended version of this paper \cite{sup}.
\end{proof}}

\ver{\begin{proof}
See the Appendix.
\end{proof}}

Theorem \ref{thm:nonexistence} implies that the following (denote as Condition A) is necessary for the existence of a resilient state estimator:
\begin{itemize}
\item[A.] The system \eqref{eq:system} is $2 \rho$-detectable.
\end{itemize}

\subsection{The Proposed Estimator}
\label{sec:design}

Assuming condition A, we now propose a resilient state estimator. 
The proposed estimator constitutes two procedures: 1) local estimation and 2) global fusion. The local estimators are defined by groups of $m-\gam$ sensors for all combinations
\[ \mathcal V \triangleq \{\I\subset \mathcal S:|\I|=m-\gam\}. \]
The number of such groups (local estimators) is $|\mathcal V | = \left( \begin{matrix} m \\ \gam \end{matrix} \right)$. 
Each local estimator $\mathcal I$ generates a state estimation $\hat x^\I$ separately based on the measurements of its sensors $y_{\mathcal I}$. In the global fusion process, the state estimate $\hat x$ is generated using the estimates from all local estimators $\mathcal I \in \mathcal V$. With slight overlap of notation, we use $\mathcal I \in \mathcal V$ to refer to a set of sensors as well as to the estimator that uses these sensors. 
Next, we outline these procedures and formally state the estimator in Algorithm 1.

\vspace{3mm}
\subsubsection{Local Estimations} From Assumption~A, for any set of sensors $\I \in \comset$, there exists a matrix $K^\I \in \mathbb R^{(m-\gam) \times n}$ such that $A + K^\I C_\I$ is strictly stable (has all eigenvectors in the open unit circle).\footnote{One way to find the matrix $K$ is via the Riccati equation, \ie, $K^\I = P C_\I^T (C_\I PC_\I^T + D_\I D_\I^T)^{-1}$, where $P$ is unique stabilizing solution of the discrete-time algebraic Riccati equation $P = A ( P - PC_\I^T (C_\I PC_\I^T + D_\I D_\I^T)^{-1} C_\I P ) A^T + B B^T$. Sufficient conditions the existence of solution $P$ is that $(A, C_\I)$ is detectable and $(A, B B^T)$ is detectable.} Using this matrix $K^\I$, we construct a local estimator that only uses measurements from the set of sensors $\I$ to produce a \textit{local state estimate} $\hat x^{\I}$:\footnote{
We use superscript notations for original vectors and matrices (\text{e.g.} $K^\I$ and $x^\I$, respectively) and subscript for vectors and matrices projected by \eqref{eq:proj_matrix} (\text{e.g.} $y_\I$ and $C_\I$, respectively).} 
\begin{align}\label{eq:est}
\hat x^{\I}(t+1) = A \hat x^\I (t)  - K^\I ( y_\I(t) -   C_\I \hat x^\I(t))
\end{align}
with the initial condition $\hat x^\I (0) = 0$. The estimation error and residual vector of \eqref{eq:est} is respectively defined as \begin{align}
\label{eq:e_def}
&e^\I(t) \triangleq x(t) - \hat x^\I(t)\\
\label{eq:r_def}
&r^\I(t) \triangleq y_\I(t) - C_\I\hat x^\I(t) 
\end{align}
The LTI system from $w$ to $e^\I$ is $E^\I(K^\I)$ defined in \eqref{eq:def_Ek}, whereas the LTI system from $w$ to $r^\I$ is $\pwe^\I(K^\I)$ defined in \eqref{eq:def_Gk}. 

When the set $\mathcal I$ does not contain any compromised sensors, \textit{i.e.}, $a_\I = 0$, the residual vector $r^\I(t)$ is determined by disturbance $w$ alone and is bounded by 
\BEQL
\label{eq:localdetection}
&&\| r^\I  \|_q \leq \| \pwe^\I(K^\I)\|_{p \rightarrow q}.
\EEQL
Condition \eqref{eq:localdetection} can only be violated when the set $\mathcal I$ contains compromised sensors, so \eqref{eq:localdetection} is a necessary condition for all the sensors in set $\I$ to be benign. The local estimator at time $t$ uses the necessary condition \eqref{eq:localdetection} to determine the validity of its estimate and label local estimator $\I$ to be \textit{invalid} upon observing
$\| r^\I(0:t) \|_q > \| \pwe^\I(K^\I)\|_{p \rightarrow q} .
$

\vspace{2mm}

\subsubsection{Global Fusion} 
From above, the set of valid local estimators $\I \in \comset(t)$ is characterized as  
\begin{align}
  \comset(t) \triangleq \left\{\I\in \mathcal S:\| r^\I(0:t) \|_q \leq \| \pwe^\I(K^\I)\|_{p \rightarrow q} \right\}.
\label{eq:comset}
\end{align}
Using $\comset(t)$, we compute the \textit{global state estimate} as follows: $\hat x(t) = [ \hat x_1(t) , \hat x_2(t) , \cdots , \hat x_n(t) ]$, where
  \begin{equation}
  \label{eq:global_fusion}
  \hat x_i(t)   =
\begin{dcases}
\frac{1}{2}    \left( \min_{\I \in \comset(t)} \hat x^\I_i (t)  + \max_{\J \in \comset(t)} \hat x^\J_i (t) \right) & q = \infty\\ 
 \frac{1}{ |\comset(t)| }  \sum_{\I(t) \in \comset(t)} \hat x^\I_i(t) & q = 2.
\end{dcases} 
\end{equation}

%\begin{remark}
%Condition \eqref{eq:localdetection} is sufficient to construct a resilient estimator. However, when $q = \infty$, the following element-wise bound can achieve improved estimation accuracy:
%\BEQ
%&&|r^\I_i(0:t)|  \leq \sum_{j = 1}^{l} \sum_{\tau = 0}^{t} \{ h_{\pwe^\I(K^\I)} (\tau) \}_{ij},
%\EEQ
%where $\{ h_{\pwe^\I(K^\I)} \}_{ij}$ is the impulse response from $w_j(t)$ to $r_i(t)$. 
%\end{remark}

\begin{algorithm}[H]
\label{alg:estimation_algorithm}
\begin{algorithmic}
\State Initialize $\comset(0) \gets \comset$ and $\hat x^\I (0)  \gets 0, \I \in \comset(0)$
\For{ $t \in \mathbb N$ }
%\State \textit{1) Local Estimations}
\For{ $\I \in \comset(t-1)$} (Local Estimation)
\State Initialize $\comset (t) \gets  \emptyset$
\State Determine $\hat x^\I(t)$ from \eqref{eq:est} and $r^\I(t)$ from \eqref{eq:r_def}
    \If {$\| r^\I (0:t) \|_q \leq \|\pwe^\I(K^\I)\|_{p \rightarrow q}$}
    	\State $\comset (t) \gets \{ \comset (t), \I \}$
    \EndIf
\EndFor
\State Obtain estimate $\hat x(t)$ from \eqref{eq:global_fusion} (Global Fusion)
\EndFor
\end{algorithmic}
\caption{The Proposed State Estimator}
\end{algorithm}

\subsection{Resilience of the Proposed Estimator}
\label{sec:est_analysis}

Previous works have shown that there exist estimators that can detect the attacks and recover the exact state for noiseless systems if the system is $2 \rho$-observable~\cite[Proposition 2]{Fawzi2014}\cite[Theorem 1]{chong2015observability}\cite[Theorem 3.2]{shoukry2016event}. In this section, we show that the proposed estimator is resilient to attack when the system is $2 \gam$-detectable.

\begin{theorem}
\label{thm:error_bound_infinity}
The estimator in Algorithm 1 has a bounded estimation error.
In particular, the estimation error is upper-bounded by
\begin{align*}
\begin{array}{lr}
 \hspace{-2mm}\underset{\I \in \comset}{\max} \|E^{\I}(K^{\I})\|_{\infty} +   \underset{\I,\J \in \comset}{\max}\sqrt{ \frac{1}{2} \log | \comset | } \; \part^{\I,\J}_{{2,2}} \hspace{-8mm}&\text{if }(p,q) = (2, 2)\\
  \hspace{-2mm}\underset{\I,\J\in \comset }{\max}\Big(  \|E^{\I}(K^{\I})\|_{2}   + {1 \over 2} \part^{\I,\J}_{{2, \infty}}\Big) &\text{if }(p,q) = (2, \infty) \\
  \hspace{-2mm} \underset{\I,\J\in \comset }{\max} \Big(  \|E^{\I}(K^{\I})\|_{1}   + {1 \over 2} \part^{\I,\J}_{{\infty, \infty}}\Big) &\text{if }(p,q) = (\infty, \infty) .
\end{array}
\end{align*}  
In the above formula, the term $\part^{\I,\J}_{{p,q}}$ is defined as 
\begin{align*}
&\part^{\I,\J}_{{p,q}} =  \alpha_{{p,q}}^{\I \cap\J}(\beta^{\I, \I \cap\J} _{{p,q}}+\beta^{\J,  \I \cap\J}_{{p,q}})\\
&  \alpha^\K_{{p,q}}  \triangleq \inf_{K:A+KC_\K\text{ strictly stable}}\left\| \left[
\begin{array}{c|c}
  A+K C_\K & \begin{bmatrix}
    I&K
  \end{bmatrix}\\
  \hline
  I &  0
\end{array}
\right]\right\|_{p \rightarrow q} \\
& \beta^{\I,\K}_{{p,q}}  \triangleq   \left\|\begin{bmatrix}
-K^\I\\
P_{\K,\I}
  \end{bmatrix}\right\|_{p \rightarrow q} \|r^\I(0:T)\|_p, 
  \end{align*} 
  where $P_{\K,\I}\in \mathbb R^{|\K|\times|\I|}$ is the unique solution of $P_\K = P_{\K,\I} P_\I$, and $\| \cdot \|_{p \rightarrow q}$ is an induced norm on matrix.

\end{theorem}

An immediate consequence of Theorem \ref{thm:error_bound_infinity} is that condition A is a necessary and sufficient condition for the construction of a resilient state estimator, and that Algorithm 1 resilient to attack. The estimation error upper-bound in Theorem \ref{thm:error_bound_infinity} decomposes into two terms: $\| E^\I(K^\I) \|_{p \rightarrow q}$ and the remaining. The first term $\| E^\I(K^\I) \|_{p \rightarrow q}$ characterize the error between a local estimator and the true state. That is, if the local estimator $\I$ is used for a system with no attack ($a \equiv 0$), then its estimation error is bounded by $\| E^\I(K^\I) \|_{p \rightarrow q}$. The second term exists due to the attack in an unknown set of sensors. When $(p,q) = (2, 2)$, the error upper-bound grows at the order $o(\sqrt{ \rho \log m} )$ for $m \rightarrow \infty$. Hence, the error can be kept small even for systems with large $m$. Moreover, it shall be noted that an increase in the tolerable number of compromised sensors $\gam$ may result in an increase in both terms, thus increasing the worst-case estimation error $\sup_{\| w \|_p \leq 1, \|a\|_0 = 0 } \| e\|_q $. 
\begin{corollary}
  A necessary and sufficient condition for the existence of an $\epsilon$-resilient estimator for some finite $\epsilon >0$ is that $(A,C_\K)$ is detectable for any index set $\K\subset \mathcal S$ with cardinality $m - 2\gam$.
\end{corollary}
\begin{corollary}
\label{thm:esterror}
Consider system (\ref{eq:system}) with $\gam$-sparse attack. The state estimator in Algorithm 1 is $\epsilon$-resilient to attack for some finite $\epsilon > 0$. 
\end{corollary}

\section{Proof for Resilience of the Proposed Estimator}
%%%%%%%%%%%% PROOF OUTLINE %%%%%%%%%%%%%%%%%%%%%%%%%%%%%%%%%%%%
We highlight important parts of the proof of Theorem \ref{thm:error_bound_infinity} in this section and present the complete proof in the extended version of this paper \cite{sup}. The proof of Theorem \ref{thm:error_bound_infinity} has two procedures: 1) bounding local estimation errors, and 2) bounding global fusion errors. Specifically, from the triangular inequality, at any time $t \in \mathbb Z_+$, the estimation error satisfies 
\begin{align}
\nonumber
  \|e \|_q
  &= \left \| \big( x - \hat x^{\I}  \big) + \big( \hat x^{\I}  - \hat x  \big) \right\|_q \\
  \label{eq:globalerror_1}
    &\leq\left \|  x - \hat x^{\I} \right \|_q + \left\| \hat x^{\I}  - \hat x  \right \|_q .
    \end{align}
    where $\I \in \mathcal B \subset \comset$ is a set that only contains benign sensors (denote $\I$ as the \textit{benign estimator}). The benign estimator $\I$ exists from assumption \eqref{eq:attack}. The first term $ \|x - \hat x^{\I}\|_q$ can be bounded using Lemma~\ref{lemma:nec_original} by
\begin{align}
\label{eq:error_firstterm}
  \|x - \hat x^{\I}\|_q \leq \|E^{\I}(K^{\I})\|_{p \rightarrow q} .
\end{align}  
Now it only remains to show that the second term is bounded. 

To bound the second term, we first bound the difference between the estimates of any two valid local estimators $\mi,\,\mj \in \comset(T)$ up to time $T \in \mathbb Z_+$, which is given in Lemma \ref{lemma:difference}. We then use Lemma \ref{lemma:difference} to show that the difference between the estimates of the benign estimator $\I \in \comset$ and the global estimator is finite. This is shown in Lemma \ref{lemma:part2_1} for $(p, q) =  (2,2)$ and in Lemma \ref{lemma:part2_2} for $(p, q) =  (2,\infty),(\infty,\infty)$. In these lemmas, each set of sensors in $\comset$ are labeled into 
\begin{align}
\J_1, \J_2, \cdots , \J_{ |\comset| }.
\end{align}
\BLEM
\label{lemma:difference}
Assume that Condition A holds. Let $\mi,\,\mj\in \mathcal V(T)$ be two sets of sensors that are valid at time $T$. The divergence between the local estimator $\mi$ and $\mj$ up to time $T$ satisfies 
\begin{align}\label{eq:local_divergence_cond}
 %\| \hat x^{\mi}(0:T) - \hat x^{\mj}(0:T) \|_q  \leq \alpha^{\mi\cap\mj} ( \beta^\mi+\beta^\mj).
 \| \hat x^{\mi}(0:T) - \hat x^{\mj}(0:T) \|_q  \leq \part^{\mi,\mj}_{{p,q}} ,
 \end{align}
where right hand side is finite, \ie, $\part^{\mi,\mj}_{{p,q}} < \infty$ . 
\ELEM

\BLEM
\label{lemma:part2_1}
If condition \eqref{eq:local_divergence_cond} holds for $(p,q) = (2,2)$ at all time $T \in \mathbb Z_+$,
then the divergence between the benign estimator $\I$ and the global estimator satisfies 
  \begin{align}  
   \|\hat x^{\I} - \hat x \|_2  \leq  \max_{\mi, \mj \in \mathcal V } \sqrt{ \frac{1}{2} \log |\mathcal V| }\; \part^{\mi,\mj}_{p,q} .
  \end{align}

\ELEM

\BLEM
\label{lemma:part2_2}
If condition \eqref{eq:local_divergence_cond} holds for $(p,q) = (2,\infty), (\infty,\infty)$ at all time $T \in \mathbb Z_+$, then the divergence between the benign estimator $\I$ and the global estimator satisfies 
  \begin{align}  
   \|\hat x^{\I} - \hat x \|_\infty  \leq   {1 \over 2} \max_{\J \in \comset} \part^{\I,\J}_{{p, \infty}} .
  \end{align}
\ELEM

\nver{ 

\noindent We prove Lemma \ref{lemma:difference}, Lemma \ref{lemma:part2_1} in Section \ref{sec:proof1}, Section \ref{sec:proof2}, respectively. The proof of Lemma \ref{lemma:part2_2} is given in the extended version of this paper \cite{sup}. 
Combining all of the above, we are now ready to prove Theorem \ref{thm:error_bound_infinity}. 

}

\ver{ 

\noindent We prove Lemma \ref{lemma:difference}, Lemma \ref{lemma:part2_1}, and Lemma \ref{lemma:part2_2} in Section \ref{sec:proof1}, Section \ref{sec:proof2}, and Section \ref{sec:proof3}, respectively.
Combining all of the above, we are now ready to prove Theorem \ref{thm:error_bound_infinity}. 

}

\BPF[Proof (Theorem \ref{thm:error_bound_infinity})]
Taking supremum over all $t \in \mathbb Z_+$ and maximizing over all sensor sets $\I$ in Lemma~\ref{lemma:part2_1}, we obtain  
\begin{align*}
\sup_{\substack{ \|w\|_2 \leq 1 \\ \|a\|_0\leq \gam}} \|e\|_2 \leq 
\underset{\I \in \comset}{\max} \|E^{\I}(K^{\I})\|_{\infty} +   \underset{\mi, \mj \in \comset}{\max}\sqrt{ \frac{1}{2} \log | \comset | } \; \part^{\I,\J}_{{2,2}} .
\end{align*} 
Applying similar argument for the case of $q = \infty$, we obtain
\begin{align*}
\|e\|_\infty \leq \max_{\I, \J \in \comset} \left(    \|E^{\I}(K^{\I})\|_{ p \rightarrow \infty} + \frac{1}{2} \part^{\I,\J}_{{p,\infty}} \right),
\end{align*}
where $p = 2, \infty$. 
\EPF

\subsection{Proof of Lemma \ref{lemma:difference}}
\label{sec:proof1}
We first present a lemma, using which we prove Lemma \ref{lemma:difference}. 

\begin{lemma}
\label{lemma:y0}
Consider system (\ref{eq:LTI}) where $(A,C)$ is detectable and $\|w \|_p \leq 1$. If $y(t) = 0$ for all $t = 0, 1, \cdots, T$, then 
\BEQL
  \label{eq:finitestatenorm}
\| x(0:T) \|_q \leq \inf_{K:A+KC\text{ strictly stable}}\| E(K) \|_{p \rightarrow q},
\EEQL
where $E(K)$ is given in \eqref{eq:def_Ek}. 
\end{lemma}

\BPF[Proof (Lemma \ref{lemma:y0})]
As $(A,C)$ is detectable, $A+KC$ is strictly stable for some matrix $K$. For such stabilizing $K$, we can construct the state estimator \eqref{eq:LTIest}. Since $y(0:T) = 0$, the state estimator \eqref{eq:LTIest} produces zero estimate $\hat x(0:T) = 0$. From Lemma~\ref{lemma:nec_original}, we obtain
\begin{equation*}
  \begin{split}
    \| x(0:T) \|_q 
   	= \|e(0:T)\|_q
   \leq \| E(K) \|_{p \rightarrow q}. 
  \end{split}
\end{equation*}
Taking infimum over all $K$ such that $A + KC$ is strictly stable, we obtain  \eqref{eq:finitestatenorm}.
\EPF

\begin{proof}[Proof (Lemma \ref{lemma:difference})]

Let $\mi,\,\mj \in \comset(T)$. We first compute the dynamics of the local estimates $\hat x^\mall(t)$, $i = 1,2$. From \eqref{eq:est} and \eqref{eq:r_def}, 
\begin{equation}
  \label{eq:suberror}
\begin{aligned}
  &\hat x^\mall(t+1) = A\hat x^\mall(t) -K^\mall r^\mall(t), \hat x^\mall(0) = 0 \\
  &y_\mall(t) = C_\mall \hat x^\mall(t) + r^\mall(t)
  \end{aligned}
\end{equation}
for $t \leq T$. We define the sequences $\phi^\mall(t)$ and $\varphi^\mall(t)$ by 
\begin{align*}
&\phi^\mall(t) \triangleq -K^\mall r^\mall(t) ,
&\varphi^\mall(t) \triangleq P_{{\K_{1,2}},\mall} r^\mall(t),
\end{align*}
where $P_{{\K_{1,2}},\mall}\in \mathbb R^{| {{\K_{1,2}}} |\times|\mall|}$ is the unique solution of $P_{{\K_{1,2}}} = P_{{\K_{1,2}},\mall} P_\mall$. Let $\mathcal K_{1,2} =  \mi \cap \mj$ the intersection between the two sets $\mi, \mj$. 
As the measurements from subset $\mathcal K_{1,2} \subset \mall$ also satisfies 
$y_{\K_{1,2}}(t) = C_{\K_{1,2}} \hat x^\mall(t) + P_{{\K_{1,2}},\mall} r^\mall(t),$
combining with \eqref{eq:suberror} yields
\begin{equation}
  \label{eq:suberror2}
  \begin{aligned}
  &\hat x^\mall(t+1) = A\hat x^\mall(t) + \begin{bmatrix} 
    I & 0 
  \end{bmatrix}\begin{bmatrix}
    \phi^\mall(t)\\
    \varphi^\mall(t)
  \end{bmatrix}, \,\hat x^\mall(0) = 0\\
  &y_{\K_{1,2}}(t) = C_{\K_{1,2}} \hat x^\mall(t) +  \begin{bmatrix} 
    0 & I 
  \end{bmatrix}\begin{bmatrix}
    \phi^\mall(t)\\
    \varphi^\mall(t)
  \end{bmatrix}.
  \end{aligned}
\end{equation}
Now, let $\Delta (t)$ be the difference between the local estimator $\mi$ and local estimator $\mj$, \ie, 
\begin{align}
\Delta (t) \triangleq \hat x^\mi(t)-\hat x^\mj(t).
\end{align}
Subtracting the equation \eqref{eq:suberror2} for $\mi$ from equation \eqref{eq:suberror2} for $\mj$, we obtain the dynamics of $\Delta $ as follows:
\begin{align*}
  &\Delta (t+1) = A\Delta (t) + \begin{bmatrix} 
    I & 0 
  \end{bmatrix}\begin{bmatrix}
    \phi^\mi(t)-\phi^\mj(t)\\
    \varphi^\mi(t)-\varphi^\mj(t)
  \end{bmatrix},\,\Delta(t) = 0,\nonumber\\
  &0 = C_{\K_{1,2}} \Delta (t) +  \begin{bmatrix} 
    0 & I 
  \end{bmatrix}\begin{bmatrix}
    \phi^\mi(t)-\phi^\mj(t)\\
    \varphi^\mi(t)-\varphi^\mj(t)
  \end{bmatrix}.
\end{align*}

\noindent Because a valid set satisfies \eqref{eq:comset}, the residual vectors of estimator $\mall$, $i = 1,2$, are bounded by 
$
 \|r^\mall(0:T)\|_q \leq\|\pwe^\mall(K^\mall)\|_{p \rightarrow q},
$
which results in 
\begin{align}
\nonumber
  \left\|\begin{bmatrix}
\phi^\mall(0:T)\\
\varphi^\mall(0:T)
  \end{bmatrix}\right\|_p 
  &\leq    \left\|\begin{bmatrix}
-K^\mall\\
P_{{\K_{1,2}},\mall}
  \end{bmatrix}\right\|_{p \rightarrow p} \|r^\mall(0:T)\|_p \\
  &= \beta^{\mall,\mi \cap \mj}_{p,q} \label{eq:phibound_I}.
\end{align}
From the triangle inequality, we obtain 
\[
  \left\| \begin{bmatrix}
\phi^\mi(0:T) - \phi^\mj(0:T) \\
\varphi^\mi(0:T) - \varphi^\mj(0:T) 
  \end{bmatrix}  \right \|_p 
    \leq \beta^{\mi,\mi \cap \mj}_{p,q}  +  \beta^{\mj,\mi \cap \mj}_{p,q} .
  \]
 Substitute $\phi^\mi-\phi^\mj$ for $u$ in Lemma~\ref{lemma:y0} and $\Delta (t)$ for $x$, we obtain  
\[ \| \hat x^{\mi}(0:T) - \hat x^{\mj}(0:T) \|_q  \leq   \alpha_{{p,q}}^{\mi \cap\mj}(\beta^{\mi, \mi \cap\mj} _{{p,q}}+\beta^{\mi,  \mi \cap\mj}_{{p,q}}).
 \]
\end{proof}

\vspace{-10mm}

\subsection{Proof of Lemma \ref{lemma:part2_1}}
\label{sec:proof2}

Define the following two optimization problems:
\begin{alignat*}{3}
	\op(n) :=  &\max_{ \s_k(i) \geq 0}  &&\sum_{i=1}^{n}\left( \frac{1}{n-i-1} \right)^2\left( \sum_{k = i}^{n} \s_k(i)\right)^2\\
  	&\;\;\mathrm{s.t.} 
     	&&\sum_{i=0}^{k} ( \s_k(i) )^2 \leq  \1, \;\;\; k =1,2,\dots,n \\
  	\opd (n) := &\;\min_{\lambda_i > 0}  &&\sum_{i=1}^{n} \lambda_i\\
  	&\;\;\mathrm{s.t.} &&\sum_{i=1}^j\frac{1}{\lambda_i}\leq (j+1)^2,\;\;\; j =1,2,\dots,n.
  \end{alignat*}
With the slight abuse of notation, we will also denote $\op(n)$, $\opd(n)$ as the optimal solutions of the optimization problem $\op(n)$, $\opd(n)$, respectively. We first show that $\op(N-1)$ with
\begin{align}
\1 =  \act
\end{align}
is an upper-bound of $\left\|  \hat x^{\I} - \hat x \right  \|_2$ (Lemma \ref{lem:2norm_1}). The problem $\op(n)$ is then converted into its dual problem $\opd(n)$, between which the duality gap is zero (Lemma \ref{lem:2norm_2}). The dual problem $\opd (n)$ admits an analytical solution that can be upper-bounded by a simple formula (Lemma \ref{lem:2norm_3}).

\begin{lemma}
\label{lem:2norm_1}
If condition \eqref{eq:local_divergence_cond} holds for $(p,q) = (2,2)$,
then the divergence between the benign estimator $\I$ and the global estimator satisfies 
  \begin{align}  \label{eq:centerl2error}
 \|\hat x^{\I} - \hat x \|_2^2  \leq \op (N-1),
  \end{align}
  where $N = |\mathcal V|$ and $\1 =  \act$. 
\end{lemma}

\nver{

 \begin{proof}[Proof (Lemma \ref{lem:2norm_1})]
 
In order to relate the infinite sequence $ \{ \hat x^{\I} (t) - \hat x (t) \}_{t \in \mathbb Z_+}$ with the finite-dimensional optimization problems $\op(n)$, we first divide the infinite time horizon into finitely many intervals. Let $T_i$ be the time the set $\I_i$ becomes invalid. Without loss of generality, we assume that $T_0 = 0$ and 
\begin{displaymath}
 T_1 \leq T_2  \leq \dots\leq T_{N-1} \leq T_N = \infty ,
\end{displaymath}
where $\I_N$ is the benign estimator ($\I_N$ only contains benign sensors). We define $\{x^{\I_i}(t)\}_{t \leq T_i}$ as a finite or sequence of length $T_i+1$ for finite $T_i$ and infinite sequence for $T_i =  \infty$.
Let a finite sequence $ \s_k(i)$, $k =  1, \cdots , N-1$, $i = 0, \cdots, k$ be defined as
\begin{displaymath}
  \s_k(i) =  \left(  \sum_{ t = T_i + 1}^{T_{i+1}} \| \hat x^{\I_N}  (t )-  \hat x^{\I_k} (t )  \|_2^2 \right)^{\frac{1}{2}} .\end{displaymath}
Using $\s_k(i)$, we obtain an upper-bound for the divergence between the benign estimator $\I_N$ and global estimator as follows:
 \begin{align*}
  &\| \hat x^{\I_N}- \hat x\|^2_2 \\
  &= \sum_{i = 0}^{N-2} \sum_{t = T_i + 1}^{T_{i+1}}  \left \|\frac{1}{N-i}  \sum_{k = i+1}^{N-1} \left( \hat x^{\I_N}(t) -  \hat x^{\I_k} (t) \right) \right\|^2_2 \\
   &\leq \sum_{i = 0}^{N-2} \left( \frac{1}{N-i} \right)^2 \left(   \sum_{k = i+1}^{N-1} \s_k(i)  \right)^2 
\end{align*} 

The sum of $i$ and $j$ counts only up to $N-1$ since $\hat x(t) =  \hat x^{\I_N}(t)$ for $t > T_{N-1}+1$. The last line is due to Cauchy Schwarz inequality. Using the above relation, $\| \hat x^{\I_N}- \hat x\|^2_2$ is upper-bounded by the optimal value of the following problem: 
\begin{equation}\label{eq:optcenterl2removeTleq}
\begin{aligned}
	 \max_{ \s_k(i) \geq 0}&\;\;
    	\sum_{i=0}^{N-2}\left( \frac{1}{N-i} \right)^2\left( \sum_{k = i+1}^{N-1} \s_k(i)\right)^2\\
  	\mathrm{s.t.}\;\;\;&\;\;
     	\sum_{i=0}^{k} ( \s_k(i) )^2 \leq \act, k =1,\dots,N-1 ,
\end{aligned}\end{equation}
which is the optimization problem $\op(N-1)$ with $\1 = \act$. 
\end{proof}
}

\ver{ 

 \begin{proof}[Proof (Lemma \ref{lem:2norm_1})]
 
In order to relate the infinite sequence $ \hat x^{\I} (t) - \hat x (t)$ with the finite-dimensional optimization problems $\op(n)$, we first divide the infinite time horizon into a finite sequence as below. Let $T_i$ be the time the set $\I_i$ becomes invalid and $T_0 = 0$. Without loss of generality, we assume that
\begin{displaymath}
  T_1 \leq T_2  \leq \dots\leq T_{N-1} \leq T_N = \infty .
\end{displaymath}
The relation $T_N = \infty$ holds because $\I_N$ is a valid from assumption \eqref{eq:attack}. We call $\I_N$ the \textit{benign estimator}. If $T_i = \infty$, then we define $\{x^{\I_i}(t)\}_{t\in \mathbb Z_+}$ as an infinite sequence of points in $\mathbb R^n$. Otherwise, if $T_i$ is finite, then we define $\{x^{\I_i}(t)\}_{t=0,\dots,T_i}$ as a finite sequence of length $T_i+1$.

Recall that $\I_N$ is the benign estimator. 
  Let $\triangle^{\I_k} (t ) = \hat x^{\I_N}  (t )-  \hat x^{\I_k} (t )$ denote the difference between the benign estimator $\I_N$ and other local estimator $\I_k$. We first define the following variable:
 \begin{displaymath}
  \s_k(i) =  \| \triangle^{\I_k} (T_i + 1: T_{i+1} ) \|_2 ,\end{displaymath}
  where $k =  1, \cdots , N-1$ and $i = 0, \cdots, k$.  

 Using $\s_k(i)$, we can bound the estimation error between the local estimator $\I_N$ and the global estimator as follows:

 \begin{align*}
  &\| \hat x^{\I_N}- \hat x\|^2_2 \\
  &= \sum_{t = 0}^\infty  \| \hat x^{\I_N}(t) - \hat x (t) \|^2_2 \\
  &= \sum_{i = 0}^{N-2} \sum_{t = T_i + 1}^{T_{i+1}}  \left \| \hat x^{\I_N}(t) - \frac{1}{|\comset(t)|} \sum_{j = i+1}^{N} \hat x^{\I_j} (t) \right\|^2_2 \\
  &= \sum_{i = 0}^{N-2} \sum_{t = T_i + 1}^{T_{i+1}}  \left \|\frac{1}{N-i}  \sum_{k = i+1}^{N-1} \left( \hat x^{\I_N}(t) -  \hat x^{\I_k} (t) \right) \right\|^2_2 \\
  &= \sum_{i = 0}^{N-2} \left( \frac{1}{N-i} \right)^2   \left \| \sum_{k = i+1}^{N-1} \triangle^{\I_k} (T_i + 1: T_{i+1} ) \right\|^2_2 \\
  &\leq \sum_{i = 0}^{N-2} \left( \frac{1}{N-i} \right)^2 \left(   \sum_{k = i+1}^{N-1}\left \| \triangle^{\I_k} (T_i + 1: T_{i+1} ) \right\|_2  \right)^2 \\
   &= \sum_{i = 0}^{N-2} \left( \frac{1}{N-i} \right)^2 \left(   \sum_{k = i+1}^{N-1} \s_k(i)  \right)^2 
\end{align*}

The sum of $i$ and $j$ counts only up to $N-1$ since $\hat x(t) =  \hat x^{\I_N}(t)$ for $t > T_{N-1}+1$. We also used Cauchy Schwarz inequality in the second to last line. Using the above relation, $\| \hat x^{\I_N}- \hat x\|^2_2$ is upper-bounded by the optimal value of the following problem: 
\begin{equation}\label{eq:optcenterl2removeTleq}\begin{aligned}
	 \max_{ \s_k(i) \geq 0}&\;\;
    	\sum_{i=0}^{N-2}\left( \frac{1}{N-i} \right)^2\left( \sum_{k = i+1}^{N-1} \s_k(i)\right)^2\\
  	\mathrm{s.t.}\;\;\;&\;\;
     	\sum_{i=0}^{k} ( \s_k(i) )^2 \leq \act, k =1,\dots,N-1 ,
\end{aligned}\end{equation}
which is the optimization problem $\op(N-1)$ with $\1 = \act$. 
\end{proof}

}

\begin{lemma}
\label{lem:2norm_2}
The problems $\op (n)$, $\opd (n)$ have identical optimal values, \ie, $\op (n) = \opd (n)$.
\end{lemma}

\noindent We use the following lemma to prove Lemma \ref{lem:2norm_2}. The proof of Lemma \ref{lem:matrix_inequality} is given in the extended version of this paper \cite{sup}.

\begin{lemma}
\label{lem:matrix_inequality}
The following two inequalities are equivalent 
\begin{align}
\label{eq:lem7_cond1}
 & \Lambda= \diag(  \lambda_1, \lambda_2 , \cdots, \lambda_n) \geq \mathbf 1\mathbf 1^T \\
 \label{eq:lem7_cond2}
 &\sum_{i=1}^n \frac{1}{\lambda_i} \leq 1 , \text{ and } \lambda_j > 0,\;\; j =1,\dots,n
\end{align}
\end{lemma}

\nver{ 

%\begin{proof}[Proof (Lemma \ref{lem:matrix_inequality})]

%\end{proof}
  
%\begin{proof}[Proof (Lemma \ref{lem:matrix_inequality})]
%  We first show that \eqref{eq:lem7_cond1} implies \eqref{eq:lem7_cond2}. 
%  %Let $e_i$ be the $i$-th canonical basis vector of $\mathbb R^n$. Observe that $\lambda_i = e_i^T \Lambda e_i \geq e_i^T \mathbf 1\mathbf 1^T e_i = 1$, and so $\Lambda$ is invertible. 
%  From \eqref{eq:lem7_cond1}, $\lambda_i \geq 1$, so $\Lambda$ is invertible. 
%  From \cite[Lemma 1.1]{Ding2007}, we have
%  \begin{displaymath}
%    \left( 1-\sum_{i=1}^n \frac{1}{\lambda_i} \right)\prod_{i=1}^n \lambda_i =  (1-\mathbf 1^T \Lambda^{-1}\mathbf 1) \det(\Lambda) 
%    = \det(\Lambda - \mathbf 1\mathbf 1^T)
%  \end{displaymath}
%  Since $\Lambda \geq \mathbf 1\mathbf 1^T$, the determinant of $\Lambda - \mathbf 1\mathbf 1^T$ is non-negative, which proves  \eqref{eq:lem7_cond2}. 
%  Similarly, applying the same argument to the principal minor of $\Lambda -\mathbf 1\mathbf 1^T$, we can show that \eqref{eq:lem7_cond2} implies \eqref{eq:lem7_cond1}.  
%  \end{proof}
}

\ver{

We will use the following lemma to prove Lemma \ref{lem:matrix_inequality}. 
\begin{lemma}[Lemma 1.1 of \cite{Ding2007}]
\label{prop:determinant}
If $ \Lambda \in \mathbb R^{n \times n}$ is an invertible matrix, and $u, v \in \mathbb R^{n}$ are two column vectors, then 
\[ \det (\Lambda + u v^T) = ( 1 + v^T \Lambda^{-1} u ) \det (\Lambda) .\]
\end{lemma}

\begin{proof}[Proof (Lemma \ref{lem:matrix_inequality})]
  We first show that \eqref{eq:lem7_cond1} implies \eqref{eq:lem7_cond2}. 
  %Let $e_i$ be the $i$-th canonical basis vector of $\mathbb R^n$. Observe that $\lambda_i = e_i^T \Lambda e_i \geq e_i^T \mathbf 1\mathbf 1^T e_i = 1$, and so $\Lambda$ is invertible. 
  From \eqref{eq:lem7_cond1}, $\lambda_i \geq 1$, so $\Lambda$ is invertible. 
  From Lemma~\ref{prop:determinant}, we have
  \begin{displaymath}
    \left( 1-\sum_{i=1}^n \frac{1}{\lambda_i} \right)\prod_{i=1}^n \lambda_i =  (1-\mathbf 1^T \Lambda^{-1}\mathbf 1) \det(\Lambda) 
    = \det(\Lambda - \mathbf 1\mathbf 1^T)
  \end{displaymath}
  Since $\Lambda \geq \mathbf 1\mathbf 1^T$, the determinant of $\Lambda - \mathbf 1\mathbf 1^T$ is non-negative, which proves \eqref{eq:lem7_cond2}. 
We then show that \eqref{eq:lem7_cond2} implies \eqref{eq:lem7_cond1}. Denote the $k$-th leading principal minor of $\Lambda -\mathbf 1\mathbf 1^T$ as $\Delta_k$. By Lemma~\ref{prop:determinant}, we obtain that
  \begin{displaymath}
   \Delta_k =  \left( 1-\sum_{i=1}^k \frac{1}{\lambda_i} \right)\prod_{i=1}^k \lambda_i \geq 0,
  \end{displaymath}
  which proves \eqref{eq:lem7_cond1}.
\end{proof}
}

\begin{proof}[Proof (Lemma \ref{lem:2norm_2})]
Let $v \in \mathbb R^{n(n-1)/2}$ be a vector that is composed of $\s_{i+1:n}(j) = \{\s_{i+1}(j),\s_{i+2}(j) , \cdots,\s_{n}(j) \} $ for all $j = 0, 1, \cdots, p$, \ie, 
  \begin{displaymath}
    v \triangleq \begin{bmatrix}
     \s_{1:n}(0) ,\s_{2:n}(1), \cdots,\s_{n}(n) \end{bmatrix} .
  \end{displaymath}
Let the following matrices be defined as
\begin{align}
\nonumber& X = vv^T \geq 0 \\
  \label{eq:def_X_F}
&F_0 = \text{diag}\left(\frac{1}{2^2},\cdots,\frac{\mathbf 1_{n-1}\mathbf 1^T_{n-1}}{(n)^2},\frac{\mathbf 1_{n}\mathbf 1^T_{n}}{(n+1)^2}\right) \\
\nonumber&F_i = \text{diag}\left(  e_{n,i } ,e_{n-1,i } ,\cdots, e_{n-i+1,i } , 0_{n-i} , \cdots , 0_{1} \right)
\end{align}
where $1_{k}$ is a $k$-dimensional vector with all elements being $1$; $e_{k,j}$ is a $k$-dimensional row vector with $j$-th entry being $1$ and other entries being $0$; and $0_k$ is a $k$-dimensional row vector with all elements being $0$. Using SDP relaxation \cite{low2013convex}, the problem $\op(n)$ can be converted into
\begin{alignat*}{3}
	 \op'(n) =  & \max_{X \geq 0}   \;\;&&\tr(F_0X)&\\
    	 		&\;\mathrm{s.t.}   \;\;&& \tr(F_iX) \leq \1,\,\forall i = 1,\dots,n &
\end{alignat*}
Hence $\op(n) \leq \op'(n)$.
%\footnote{It is possible to show the relaxation is exact, \ie, $\op(n) = \op'(n)$. This can be proved by realizing that one solution of $\op'(n)$ is a positive definite rank-one matrix $X$.}} 
This relaxation can be observed from the following relations: 
\begin{align*}
& \sum_{i=0}^{n-1}\left( \frac{1}{n-i+1} \right)^2\left( \sum_{k = i+1}^{n} \s_k(i)\right)^2 = v^T F_0 v= \tr(F_0v v^T) \\
& \sum_{i=0}^{k} ( \s_k(i) )^2 = v^T F_i v = \tr(F_i v v^T) .
\end{align*}
 We next show that the relaxation of the problem $\op(n)$ to the semidefinite problem $\op'(n)$ is also exact. Assume that $\op (n)$ is feasible and bounded. Let $X^*= \{ x^*_{ij}\}$ be the optimal solution of $\op'$. Define $X$ as 
 %It can be shown that there exists a rank-one positive semidefinite matrix $X$ such that $\tr(F_0X) = \tr(F_0X^*)$ and $\tr(F_iX) = 1$, $i=1,\dots,n$. We first define $X$ as
  \begin{displaymath}
    X = \begin{bmatrix}
      \sqrt{x_{11}^*}&\dots&\sqrt{x_{nn}^*}
    \end{bmatrix}^T \begin{bmatrix}
      \sqrt{x_{11}^*}&\dots&\sqrt{x_{nn}^*}
    \end{bmatrix} .
  \end{displaymath}
 From $x_{ii} = x^*_{ii}$ and \eqref{eq:def_X_F}, $X$ satisfies the contraints
\begin{align}
\label{eq:rem2_1}
\tr(F_iX) = \tr(F_iX^*) = \1.
\end{align} 
Furthermore, because $X^*$ is the optimal solution and $x_{ij} = \sqrt{x_{ii}^*x_{jj}^*}\geq x_{ij}^*$ (due to $X^* \geq 0$),
\begin{align}
\label{eq:rem2_2}
\tr(F_0X^*) \geq \tr(F_0X) \geq \tr(F_0X^*) . 
\end{align} 
Therefore, \eqref{eq:rem2_1} and \eqref{eq:rem2_2} shows that $\op (n) = \op'(n)$.

Next, we consider the following dual problem of $\op'(n)$: 
  \begin{align}
    \label{eq:sdpdual}
  \opd' (n) = \min_{\lambda} \; & \;   \1 \sum_{i=1}^{n} \lambda_i \\
    \mathrm{s.t.} \;\; &\sum_{i=1}^{n} \lambda_i F_i\geq F_0. \nonumber
  \end{align}
  Because there exists an strictly positive definite matrix $X > 0$ such that $\tr(F_i X) =\1$ for all $i = 1, 2, \cdots, n$, from Slater's condition~\cite{Boyd2004}, strong duality holds between $\op'(n)$ and $\opd'(n)$. Let $ \Lambda_{1:i} = \diag(\lambda_1,\dots,\lambda_i)\in \mathbb R^{i\times i}$, and observe that 
$
    \sum_{i=1}^{n} \lambda_i F_i = \diag(\Lambda_{1:1},\Lambda_{1:2},\dots,\Lambda_{1:n}). 
$
  Hence, the constraint $\sum_{i=1}^{n}\lambda_i F_i \geq F_0$ is equivalent to
  \begin{align*}
    & \diag(\lambda_1,\dots,\lambda_j) \geq \frac{1}{(j+1)^2} \mathbf 1\mathbf 1^T ,\,\forall j=1,\dots,n,  \\
    & \Longleftrightarrow \sum_{i=1}^j\frac{1}{\lambda_i}\leq (j+1)^2,\;\; \lambda_j > 0,\;\; j =1,\dots,n, 
\end{align*}
where the second line is due to Lemma~\ref{lem:matrix_inequality}. Therefore, the dual problem $\opd'(n)$ can be reformulated into $\opd(n)$. 
\end{proof}

 \begin{lemma}
  \label{lem:2norm_3}
   The solution of the problem $\opd (n)$ satisfies
  \begin{equation}
    \opd(n)  = \1 \left\{ \frac{1}{4} + \sum_{i = 2}^n \frac{1}{2 i + 1} \right \} \leq  \frac{1}{2} \1 \log(n+1) .
    \label{eq:oddharmonicseries}
  \end{equation} 
\end{lemma}

\nver{
The proof of Lemma \ref{lem:2norm_3} is given in the extended version of this paper \cite{sup}. Now we are ready to prove Lemma \ref{lemma:part2_1}. 
%\begin{proof}[Proof (Lemma \ref{lem:2norm_3})]
%We first consider when $n \geq 2$. 
%One can prove that the following solution is feasible for $D_\delta(n)$:
%  \begin{align}
%  \label{eq:opt_sol_D}
%  \lambda_i = \begin{cases} 1/4  & \text{ if } i =1 \\ 
%  1/(2i+1) & \text{ if } i\geq 2.
%  \end{cases}
%\end{align}
%Conversely, it can also be verified using Karush-Kuhn-Tucker (KKT) conditions that \eqref{eq:opt_sol_D} is indeed optimal \cite{sup}. 
%
%Next, we upper-bound the optimal cost $\opd'(n)$. By the convexity of $1/x$, we have 
%\begin{align*}
% (2 i + 1)^{-1}  \leq \frac{1}{2} ( \log (i+1)-\log (i) ).
%\end{align*}
%Combining above bound with \eqref{eq:opt_sol_D}, we establish 
% \begin{align*}
%   \opd(n) &= \1 \left[ \frac{1}{4} +  \sum_{i=1}^n\frac{1}{2i+1} -  \frac{1}{3} \right]
%   \leq \frac{1}{2}  \1  \log(n+1) 
% \end{align*}
% On the other hand, when $n = 1$, the optimal variable is $\lambda^*_1 = \1 /4$, which attains the optimal value $ \opd(1) = \1 /4 \leq \1 \log(  2 ) / 2 \approx  \1  \; 0.346574 $. 
%\end{proof}

}

\ver{
\begin{proof}[Proof (Lemma \ref{lem:2norm_3})]
We first consider the case when $n \geq 2$. Let us define an auxiliary variable
\begin{align}
  & s_k = \sum_{i=1}^k \frac{\1}{\lambda_i} ,  & k = 1, \cdots, n,
\end{align}
where $s_0 = 0$, and $s_{k} = 0$ for $k \geq n+1$. Using the auxiliary variable $s_k$, we rewrite the optimization problem $\opd(n)$ as 
  \begin{equation}
   \label{eq:s_inv_problem}
  \begin{aligned}
    \mathrm{min}\;\;& \sum_{i=1}^n\frac{1}{s_i-s_{i-1}} \\
   \mathrm{s.t. } \;\;& s_i\leq (i+1)^2,\;\;\; s_{i-1} \leq s_{i}, \;\;\; i=1,\dots,n
  \end{aligned}
  \end{equation}
We define the Lagrangian $L: \mathbb R^n \times \mathbb R^n \times \mathbb R^n \rightarrow \mathbb R$ of the problem \eqref{eq:s_inv_problem} as follows: 
   \begin{displaymath}
    L (s, \mu, \eta) = \sum_{i=1}^n \frac{1}{s_i-s_{i-1}} + \mu_i (s_i-(i+1)^2) + \eta_i(s_{i-1}-s_{i}).
 \end{displaymath}
    Let $(s^*, \mu^*, \eta^*)$ be any optimal primal and dual variables, then $(s^*, \mu^*, \eta^*)$ satisfies the Karush-Kuhn-Tucker (KKT) conditions. Solving the KKT conditions, we obtain that 
  \begin{align*}
&  s_i^* = (i+1)^2, \;\;\;\;\;\;\;\; \;\;\;\;\;\;\;\;\;\;\;\;\;\;\;\;  i=1,\dots,n \\
& \eta_i^* = 0, \;\;\;\;\;\;\;\; \;\;\;\;\;\;\;\;\;\;\;\;\;\;\;\;\;\;\;\;\;\;\;\;\;\;    i=1,\dots,n\\
 &   \mu_i^* = \begin{cases}
      \frac{1}{16}-\frac{1}{25}&\text{ if }i=1\\
      \frac{1}{(2i+1)^2} -\frac{1}{(2i+3)^2} &\text{ if }i=2,\dots,n-1\\
      \frac{1}{(2n+1)^2}&\text{ if }i=n
    \end{cases}
\end{align*}
To see this, observe that 
 \begin{enumerate}
  \item $s_i^* -  (i+1)^2\leq 0$, $\mu_i^*(s_i^* - (i+1)^2) = 0$, $i=1,\dots,n$.
  \item $s_{i-1}^* -  s_{i}^*\leq 0$, $\eta_i^*(s_{i-1}^* - s_i^*) = 0$, $i=1,\dots,n$.
  \item $\mu_i^*  \geq 0$, $i=1,\dots,n$
  \item $\eta_i^* \geq 0$, $i=1,\dots,n$.
  \item     
   $
      d L(s^*,\mu^*,\eta^*) /ds_i  = 0
    $
\end{enumerate}
Since the problem \eqref{eq:s_inv_problem}, which is equivalent with $\opd(n)$, is a convex problem, the KKT conditions are sufficient for optimality~\cite{Boyd2004}. Therefore, the optimal primal and dual variables are $(s^*, \mu^*, \eta^*)$, which result in the optimal value 
\begin{align}
\label{eq:Pn_formula}
\opd(n)  = \frac{1}{4} + \sum_{i=2}^n  \frac{1}{2i+1} .
\end{align}
Next, we upper-bound the optimal cost $\opd'(n)$. Since the function $f ( i )  = (2 i + 1)^{-1}$, $ i \geq 1$, is convex, from Jensen's inequality, $f ( i )$ is upper-bounded by
$f ( i ) \leq \frac{1}{2}\left( f( i - t/2 ) +f( i + t/2 )\right),$ for any $t \in [ 0, 1]$. Integrating this along $t \in [ 0 ,1]$, we obtain that 
\begin{align*}
 f ( i ) &\leq \frac{1}{2} \int_{t=0}^1 f \left( i - \frac{t}{2} \right) +f\left( i + \frac{t}{2} \right)dt \\
 &= \frac{1}{2} ( \log (i+1)-\log (i) ).
\end{align*}
Combining above bound with \eqref{eq:Pn_formula}, we establish a lower-bound of the optimal value
 \begin{align*}
   \opd(n) &=  \1 \left\{ \frac{1}{4} +\sum_{i=1}^n\frac{1}{2i+1} -  \frac{1}{3} \right\}
   \leq  \frac{1}{2} \1 \log(n+1) 
 \end{align*}
 On the other hand, when $n = 1$, the optimal variable is $\lambda^*_1 = \1/4$, which attains the optimal value $ \opd(1) = \1/4 \leq \log(  2 ) / 2 \approx \1 0.346574$. 
\end{proof}

}

\BPF[Proof (Lemma \ref{lemma:part2_1})]
Applying Lemma~\ref{lemma:difference} and Lemma \ref{lem:2norm_1}, Lemma \ref{lem:2norm_2}, and then Lemma \ref{lem:2norm_3} consecutively, we obtain 
\begin{align*}
\|\hat x^{\I} - \hat x \|^2_2 \leq \op = \opd 
 \leq \frac{1}{2} \log( N ) \act. 
\end{align*}
\EPF

\ver{

\subsection{Proof of Lemma \ref{lemma:part2_2}}
\label{sec:proof3}
Lemma \ref{lemma:part2_2} is a trivial extension of the following Proposition. We omit its prove due to space constraints. 
\begin{proposition}
  \label{lemma:mediandivergence}
Let $z_1,\dots,z_l$ be real numbers. Define
\begin{align*}
  z = \frac{1}{2}\left(\max_i z_i + \min_i z_i\right).
\end{align*}
Then for any $i$, we have
\begin{align}
  |z - z_i| \leq \frac{1}{2} \max_j |z_j-z_i|. 
  \label{eq:mediandivergence}
\end{align}
\end{proposition}
}
\vspace{-5mm}

\section{Numerical examples}

In this section, we study the proposed estimator numerically and compare it with existing algorithms from Shoukry et al.~\cite{shoukry2017secure}, Chong et al.~\cite{chong2015observability}, Pajic et al.~\cite{pajic2014robustness}, and Lu et al.~\cite{lu2017secure}. We tested the IEEE 14-Bus system, the Unmanned Ground Vehicle (UGV), and the temperature monitor as follows.
\begin{itemize}
\item[(i)] \textit{IEEE 14-Bus system~\cite{zimmerman2011matpower,liu2009,shoukry2017secure}:} The IEEE 14-Bus system is modeled as the system \eqref{eq:system} with $A$ and $C$ given in \cite{zimmerman2011matpower}. We additionally add process noise and sensor noise by setting $B = \begin{bmatrix} I_{10} & O_{10,35} \end{bmatrix}$ and $D = \begin{bmatrix} O_{10,35} & I_{35} \end{bmatrix}$. 
\item[(ii)] \textit{Unmanned ground vehicle (UGV)~\cite{pajic2014robustness,shoukry2017secure}:} A UGV moving in a straight line has the dynamics 
 \begin{align*}
 \begin{bmatrix} \dot p \\ \dot v  \end{bmatrix} =    \begin{bmatrix} 0 & 1 \\ 0 & -b/m \end{bmatrix} \begin{bmatrix} \dot x \\ \dot v  \end{bmatrix} + \begin{bmatrix} 0 \\ 1/m  \end{bmatrix} u +  \begin{bmatrix} I_2 & O_{2,3} \end{bmatrix} w ,
\end{align*}
where $p$ is the position, $v$ is the velocity, $u$ is the force input, $w$ is the disturbance, $m$ is the mechanical mass, $b$ is the translational friction coefficient, $I_n$ is a $n$-dimensional identity matrix, and $O_{n,m}$ is a $m \times n$ zero matrix. We assume that the estimator can access the values of $u$. The UGV is equipped with a sensor measuring $x$ and two sensors measuring $v$, \ie, 
 \begin{align*}
y =    \begin{bmatrix} 1 & 0 \\ 0 & 1 \\ 0 & 1 \end{bmatrix} \begin{bmatrix}p \\ v  \end{bmatrix} + \begin{bmatrix} O_{3,2} & I_2 \end{bmatrix} w .
\end{align*}
The system parameters are the same as~\cite{shoukry2017secure}: $m = 0.8kg$, $b = 1$, and sampling interval $T_s = 0.1s$. 
\item[(iii)] \textit{Temperature monitor~\cite{mo2011sensor}:} The heat process in a plannar closed region $(z_1, z_2 ) \in [0, l ] \times  [0, l ]$ can be expressed by 
\begin{align}
\label{eq:heat}
\frac{ \partial x } {\partial t } = \alpha \left( \frac{\partial^2 x } {\partial z^2_1 }  + \frac{\partial^2 x } {\partial z^2_2 } \right)  , 
\end{align}
where $\alpha$ is the speed of the diffusion process; and $x(z_1,z_2)$ is the temperature at position $(z_1, z_2 )$ subject to the boundary conditions 
\begin{align*}
\frac{ \partial x } {\partial z_1 } \bigg|_{t , 0, z_2 }= \frac{ \partial x } {\partial z_1 } \bigg|_{t , l, z_2 }= \frac{ \partial x } {\partial z_1 } \bigg|_{t , z_1, 0 }= \frac{ \partial x } {\partial z_1 }\bigg|_{t , z_1, l } = 0 .
\end{align*}
We discretize the region using a $N \times N$ grid and the continuous-time with sampling interval $T_s$ to model \eqref{eq:heat} into \eqref{eq:system}. We additionally add process noise and sensor noise by setting $w $, $B = \begin{bmatrix} I_{9} & O_{9,20} \end{bmatrix}$ and $D = \begin{bmatrix} O_{9,20} & I_{20} \end{bmatrix}$. We set $\alpha = 0.1m^2/s$, $l = 4m$, and $N = 5$ as in~\cite{mo2011sensor}.
\end{itemize}

The noise $w$ is generated from a uniform distribution between $[-1,1]$. The time horizon is set to be $T = 100$. The number of compromised sensors is set to be $1$; the compromised sensor is randomly chosen among non-critical sensors (\ie, sensors that can be removed without losing observability). The attack signal is drawn from a Gaussian distribution with mean zero and variance $10^4$ and $1$. For the proposed algorithm, we used \eqref{eq:localdetection} with $(p,q) = (2,2), (2, \infty)$, and $(\infty, \infty)$ simultaneously. We ran each example for $100$ times and recorded their average estimation errors $\| e \|_2$ and computation times. 
%The computer used is Windows Intel Core i5-4690 Processor (4x3.50GHz/6MB L3 Cache), and the program used is MATLAB. 
The code is written in Matlab (Windows) and runs on an Intel Core i5-4690 Processor (4x3.50GHz/6MB L3 Cache). We summarize the estimation errors and computation times in Table \ref{tab:errorA100} and Table \ref{tab:compA100}. Some entries are left as NA (not applicable) because the system (ii) does not satisfy the linear matrix inequality (LMI) assumption required by the algorithm proposed by \cite{lu2017secure}. The algorithms tested have different relative accuracies and computation times from example to example. Among all examples tested, the proposed algorithm has relatively low estimation errors and average computation times.\vspace{-3mm}

\begin{table}
\begin{center}
\begin{tabular}{ l | c c c c c c}
& Proposed & \cite{shoukry2017secure} & \cite{chong2015observability} & \cite{pajic2017design} &\cite{lu2017secure} \\
\hline
 (i)  & 11.1573 &  17.2948 &  13.7927  & 48.9880 & 17.5957 \\
 (ii) & 6.9108  &  4.7246 &   7.2937 &  22.0884 &  NA  \\
  (iii) & 6.7833    &7.9424 &   6.8902   & 8.5448 & 18.5803 \\
\end{tabular}
\caption{Estimation errors in two-norm $\| e(1:100) \|_2$ when the attack variance is $10^4$. }\label{tab:errorA100}

\begin{tabular}{ l | c c c c c c}
& Proposed & \cite{shoukry2017secure} &  \cite{chong2015observability} &  \cite{pajic2017design}  & \cite{lu2017secure} \\
\hline
 (i) & 0.0062  &  0.0108  &  0.0045  &  0.0006 &  0.0345  \\
    (ii)&  0.0001 &   0.0026   & 0.0002  &  0.0003 & NA \\ 
  (iii) & 0.0023   & 0.0096  &  0.0020   & 0.0005 &  0.0048 \\
\end{tabular}
\caption{Average computation time in second for time horizon $100$ when the attack variance is $10^4$.}\label{tab:compA100}
\end{center}
\vspace{-1cm}
\end{table}

\ver{

\begin{table}
\vspace{10mm}
\begin{center}
\begin{tabular}{ l | c c c c c c}
& Proposed & \cite{shoukry2017secure} &  \cite{chong2015observability} &  \cite{pajic2014robustness}  & \cite{lu2017secure} \\
\hline
 (i)  & 6.4653  &  9.6477  &  7.4754  &  9.8822 &  17.3752 \\
 (ii) & 6.8500   & 4.6481   & 7.2884  & 23.3668 &  NA \\
  (iii) &6.7060  &  8.0070   & 6.8144  &  8.6655 & NA\\
\end{tabular}
\caption{Estimation errors in two-norm, \ie, $\| e(1:100) \|_2$ when the attack variance is $1$. }\label{tab:errorA1}
\vspace{3mm}
\begin{tabular}{ l | c c c c c c}
& Proposed & \cite{shoukry2017secure} &  \cite{chong2015observability} &  \cite{pajic2014robustness}  & \cite{lu2017secure} \\
\hline
 (i) & 0.0062  &  0.0109  &  0.0046  &  0.0005 &  0.0325 \\
    (ii)&  0.0001 &   0.0024  &  0.0002  &  0.0003 &  NA\\ 
  (iii) & 0.0023   & 0.0093   & 0.0020  &  0.0004 &   NA \\
\end{tabular}
\caption{Average computation time in second for time horizon $100$ when the attack variance is $1$. }\label{tab:compT1}
\end{center}
\vspace{-1cm}
\end{table}
}

\section{Conclusion}
 
In this paper, we propose a real-time state estimator for noisy systems that is resilient to sparse sensor integrity attacks. The proposed estimator is stable under relaxed assumptions. Its worst-case estimation errors are $O ( \log { {m}\choose{\rho} })$ in the $\mathcal H_2$ system, $O( 1 )$ in the $\mathcal H_\infty$ system, and $O( 1 )$ in the $\ell_1$ system. Its computational complexity is $O ( \log { {m}\choose{\rho} })$. 

%\begin{figure}[ht]
%  \begin{center}
%\input{myfile.tex}
% \end{center}
% \vspace{-5mm}
%  \caption{The evolution of the estimation error evolution. The disturbance $w(k)$ is generated from a uniform distribution on support $[-1,1]$ and attack $a(t) = [t*v(t), 0 , 0]'$, where $v(t)$ is independent and standard normally distributed.  The dynamics of the regular estimator is 
%  $\hat x(t+1) =  \hat x(t) +\frac{1}{3}
%    [1 \; 1\; 1] ( y(t) - C\hat x(t)),
%$ which minimizes $\max_{ \| w \|_\infty \leq 1} \| e \|_\infty$ among estimators of the form \eqref{eq:LTIest} without the consideration of the attack.}
%\label{fig:estimationerror0}
%\end{figure}

\appendices
\ver{
\section{Proof of Theorem \ref{thm:nonexistence}}

An intermediate step in the proof of Theorem \ref{thm:nonexistence} is to consider the following condition (denote as Condition B). 
\begin{itemize}
\item [B.] There exist two states $(x,x')$, two disturbances $(w,w')$, and two attacks $(a,a')$ such that all of the followings are satisfied: 
\begin{itemize}
\item[a)] both $(x, w, a)$ and $(x', w', a' )$ satisfies the dynamics \eqref{eq:system} and assumptions $\| a \|_0 \leq \gamma$, and $\| w \|_p \leq 1$
\item[b)] $y(t) = y(t)'$ at all time $t \in \mathbb N$
\item[c)] the difference between the two states is unbounded, \ie, $\| x - x'\|_q = \infty$. 
\end{itemize}
\end{itemize}

\begin{lemma}
\label{lem:limit_1}
If $(A, C_\K)$ is not detectable for some set $\K \subset \mathcal S $ with $|\K| = m - 2 \gamma$, then Condition B holds. 
\end{lemma}

\begin{proof}[Proof (Lemma \ref{lem:limit_1})]
We first prove that an undetectable $(A, C_\K)$ implies that the linear transformation $\mathcal O_t: \mathbb R^n \rightarrow  \mathbb R^{t (m - 2 \gamma) }$ defined by  
\begin{align*}
\mathcal O_t = \begin{bmatrix}
C_\K \\
C_\K A \\
C_\K A^2 \\
\vdots \\
C_\K A^{t-1} 
\end{bmatrix}
\end{align*}
has a non-trivial kernel (Step 1). Form the kernel space of $\mathcal O_t$, we then find two states $(x, x')$ that satisfy condition B (Step 2). 

Step 1: If $(A, C_\K)$ is not detectable for some set of sensors $\K$, then at least one of the following conditions holds. 
\begin{itemize}
\item [(i)] For some $z \in \mathbb R$, $\text{abs}(z) \geq 1$ and $v \in \mathbb R^n$, $A v = z v$ and $C v = 0$. 
\item [(i)] For some complex conjugate pairs $z,\bar z \in \mathbb C$, $\text{abs}(z) = \text{abs}(\bar z) \geq 1$ and $v, \bar v \in \mathbb C^n$, $A v = z v$, $A \bar v = \bar z \bar v$, $C v = 0$, and $C \bar v = 0$ . 
\end{itemize}
Condition (i) implies that 
\begin{align}
\label{eq:o_r1}
& \mathcal O_t v = 0, & t \in \mathbb N
 \end{align}
 and Condition (i) implies that 
 \begin{align}
 \label{eq:o_r2}
& \mathcal O_t (v + \bar v ) = 0,   & t \in \mathbb N.
 \end{align}

Step2: We construct two dynamics with the same measurement. There exists two disjoint sets of sensors $\mathcal K_1, \mathcal K_2$ that satisfy $|\mathcal K_1| = |\mathcal K_2| = \gamma$, $\mathcal K_1 \cap \mathcal K_2 \cap \mathcal K = \emptyset$, and $\mathcal K \cup \mathcal K_1 \cup \mathcal K_2 = \mathcal S$. Consider first when condition (i) holds. Since $B$ has full-rank, there exists an impulse disturbance $w(0)$ that produces
\begin{equation} 
\label{eq:syste_ex1}
\begin{aligned}
&x(1) = v, \\
&w(t) = 0, t \geq 1 \\
&a_i(t) = \begin{cases}
- C_i x(t) &i\in \mathcal K_1\\
0 & \text{Otherwise}
\end{cases}
\end{aligned}
\end{equation}
and 
\begin{align}
\label{eq:syste_ex2}
\begin{aligned}
&x'(1) = 0\\
&w'(t) = 0 ,  t \geq 1\\ 
&a'_i(t) = \begin{cases}
C_i x(t) &i\in \mathcal K_2\\
0 & \text{Otherwise},
\end{cases}
\end{aligned}
\end{align}
where $C_i$ denotes the $i$-th row of sensing matrix $C$. Using \eqref{eq:o_r1}, we can show that measurement $y(t)$ under \eqref{eq:syste_ex1} and $y'(t)$ under \eqref{eq:syste_ex2} are identical. However, the state under \eqref{eq:syste_ex1} is $x(t) = z^t v$, the state under \eqref{eq:syste_ex1} is $x'(t) = 0$, and thus their difference $\| x - x' \|_{q}$ is unbounded. Consider next when condition (ii) holds. There exists an impulse disturbance $w(0)$ that achieves $x(1) = v+\bar v$, so let $x(1) = v+\bar v$ replace $x(1)$ in \eqref{eq:syste_ex1}. Similarly, we can derive from \eqref{eq:o_r2} that measurements $y(t)$ and $y'(t)$ are identical, but $x(t) = z^t v +  \bar z^k \bar v$ and $x'(t) = 0$, yielding unbounded $\| x - x' \|_{q}$. Since at least (i) or (ii) holds, we have proved Condition B. 

\end{proof}

\begin{lemma}
\label{lem:limit_2}
If Condition E holds, then a resilient estimator cannot be constructed.   
\end{lemma}
\begin{proof}[Proof (Lemma \ref{lem:limit_2})]
Let $\hat x$ be the state estimation of any estimator when measurement $y = y'$ is observed. From the the triangle inequality, the estimation error $e = x - \hat x$ under $(x,w,a,y)$ and error $e' = x' - \hat x$ under $(x',w',a',y')$ satisfies
\begin{align}
\| x - x' \|_q &\leq  \| e \|_q  +  \| e' \|_q  
\end{align}
This suggests that either $\| e \|_q$  or $\| e' \|_q$ are unbounded. Because
\begin{align}
 \sup_{\|w\|_p\leq 1,\,\|a\|_0\leq \gamma} \|e\|_q \geq  \max \{ \| e \|_q ,  \| e' \|_q \} ,
 \end{align}
no estimator can achieve bounded worst-case estimation error.  
\end{proof}

\begin{proof}[Proof (Theorem \ref{thm:nonexistence})]
From Lemma \ref{lem:limit_1}, if $(A, C_\K)$ being not detectable for some set $\K \subset \mathcal S $ with $|\K| = m - 2 \gamma$, then Condition B holds. However, due to Lemma \ref{lem:limit_2}, Condition B implies that no resilient estimator can be constructed.
\end{proof}

}

%%%%%%%%%%%%%%%%%%%%%%%%%%%%%%%%%%%

\section*{Acknowledgment}

The authors would like to thank Professor John Doyle and Professor Richard Murray for insightful discussions. 

%\ifCLASSOPTIONcaptionsoff
%  \newpage
%\fi

\bibliographystyle{IEEEtran}
\bibliography{security,bib}
%
%\begin{IEEEbiography}[{\includegraphics[width=1in,height=1.25in,clip,keepaspectratio]{yorie}}]{Yorie Nakahira}
%received the Bachelor of Engineering degree from the Department of Control and Systems Engineering,
%Tokyo Institute of Technology, Tokyo, Japan, in 2012. She
%is currently a Ph.D. candidate in the Department of Computing and Mathematical Sciences, California Institute of Technology, Pasadena, CA, USA. Her research interests include robust control and networked control systems.
%\end{IEEEbiography}
%
%\begin{IEEEbiography}[{\includegraphics[width=1in,height=1.25in,clip,keepaspectratio]{yilin}}]{Yilin Mo}
%received the Bachelor of Engineering degree from the Department of Automation, Tsinghua University, Beijing, China, in 2007, Ph.D. from the Department of Electrical and Computer Engineering Department, Carnegie Mellon University, Pittsburgh, PA, in 2007. He is currently a assistant Professor in the Department of Electrical and Electronic Engineering, Nanyang Technological University. His research interests include secure control systems and networked control systems, with applications in sensor networks and power grids.
%\end{IEEEbiography}

\end{document}